\documentclass[preprint,12pt,authoryear]{elsarticle}

\usepackage{amsthm}
\usepackage{amsmath}
\usepackage{amssymb}
\usepackage{cleveref}
\usepackage{graphicx}
\usepackage{xfrac}
\usepackage{amsfonts}
\usepackage{amsmath}
\usepackage{amssymb}
\usepackage{eurosym}

\usepackage{enumerate}
\usepackage{color}
\usepackage{cleveref}
\usepackage{epstopdf}
\usepackage{dsfont}

\newtheorem{theorem}{Theorem}[section]
\newtheorem{lemma}[theorem]{Lemma}

\newtheorem{remark}[theorem]{Remark}
\newtheorem*{question*}{Open Question}
\numberwithin{equation}{section}

\journal{Applied Mathematical Letters}

\begin{document}

\begin{frontmatter}


\ead{andrei.stan@ubbcluj.ro}

\title{Role of partial functionals in the study of variational systems}


\author{Andrei Stan}

\affiliation{organization={{Faculty of Mathematics and Computer Science, Babeș-Bolyai
University}},
            city={Cluj-Napoca},
            postcode={400084}, 
            country={Romania}}

            \affiliation{organization={Tiberiu Popoviciu
Institute of Numerical Analysis, Romanian Academy},
            city={Cluj-Napoca},
            postcode={400110}, 
            country={Romania}}
\begin{abstract}

Applying techniques originally developed for systems lacking a variational structure, we establish conditions for the existence of solutions in systems that possess this property but their energy functional is unbounded both lower and below. We show that, in general, our conditions differ from those in the classical mountain pass approach by Ambroseti-Rabinovitz when dealing with systems of this type.
Our theory is put into practice in the context of a coupled system of Stokes equations with reaction terms, where we establish sufficient conditions for the existence of a solution.
The systems under study are   intermediary between gradient-type systems and Hamiltonian systems.
\end{abstract}



\end{frontmatter}

 \section{Introduction and Preliminaries}

Many real-world processes can be represented by equations or systems of equations. However, solving these problems can be quite challenging. Over time, various techniques have been developed, with the critical point technique being one of the most significant. This technique is important because it simplifies the task of solving an equation to demonstrating that a specific function has a critical point. 

In the recent papers \cite{p_a_linking, p, s, s_jnfa}, systems of the form \begin{equation*}
    \begin{cases}
        E_{11}(u,v)=0 \\
        E_{22}(u,v)=0,
    \end{cases}
\end{equation*}
were considered, where $E_1, E_2$ are certain $C^1$ functionals.  Such systems have the property that they lack a variational structure as a whole but possess it individually on each component. 

In this paper we consider  systems of the form \begin{equation}\label{sistem abstract}
    \begin{cases}
        E_{u}(u,v)=0 \\
        E_{v}(u,v)=0,
    \end{cases}
\end{equation}
where $E$ is a $C^1$ functional. In the literature there are many tools to establish the existence of critical points for $E$. However, if $E$ has no upper and lower bounds, or is not well behaved, such methods may fail. Our aim is to use the techniques developed in \cite{p_a_linking, p} to prove the existence of critical points for $E$, using some partial functionals  $E_1$, $E_2$, which may not necessarily be related to $E$. 

The novelty of this paper consist in obtaining different conditions then the ones  typically used in the classical mountain pass approach by Ambroseti-Rabinovitz for the existence of a solution for the system \eqref{sistem abstract}.

Our theorey is applied to an abstract system from $H_0^1(\Omega)$ as well as a system of Stokes equations. The latter system comes in the study of fluid dynamics and it is obtained neglecting the nonlinear term from the Navier-Stokes equations, which is an Agmon-Douglis-Nirenberg elliptic and linear system. We send to \cite{galdi2011introduction, fara_p,sohr2001navierstokes} for further details.

In the following section, we will review some important results from functional analysis, matrices converging to zero, and the Stokes system. These concepts will be used in the upcoming material.
\subsection{Ekeland variational principle}

The proof of our main result (\Cref{teorema principala}) is essentially based on the weak form of
Ekeland's variational principle (see, e.g., \cite{f}).

\begin{lemma}[Ekeland Principle - weak form]
\label{ekeland} Let $(X,d)$ be a complete metric space and let $\Phi
:X\rightarrow \mathbb{R} \cup \{+\infty \}$ be a lower semicontinuous and
bounded below functional. Then, given any $\varepsilon >0$, there exists $%
u_{\varepsilon }\in X$ such that
\begin{equation*}
\Phi (u_{\varepsilon })\leq \inf_{X}\Phi +\varepsilon
\end{equation*}
and
\begin{equation*}
\Phi (u_{\varepsilon })\leq \Phi (u)+\varepsilon d(u,u_{\varepsilon }),
\end{equation*}
for all $u\in X.$
\end{lemma}

\subsection{Abstract linear operator}\label{definitie A}

Let $\Omega\subset \mathbb{R}^n$, $n\geq 3$, be a bounded open set. Let $A\colon H_0^1(\Omega)\to H^{-1}(\Omega)$ be a continuous and strongly monotone operator, that is, there exists $\theta>0$ such that  \begin{equation}\label{tare monotonie L}
    \langle Au,u \rangle \geq  \theta |u|_{H_0^1} ^2, \text{ for all }u\in H_0^1(\Omega).
\end{equation} Here, $ \langle \cdot,\cdot \rangle$ stand for the dual pairing between $H^{-1}(\Omega)$ and $H_0^1(\Omega)$.  
We observe that for every $h\in H^{-1}(\Omega)$, Riesz representation theorem guarantees that there exists a unique element $u_h\in H_0^1(\Omega)$ such that \begin{equation*}
    A u_h=h,
\end{equation*}
i.e., $A$ is a bijective, where $A^{-1}h=u_h$.  If $h\in L^2(\Omega)$, we have that \begin{equation*}
    \langle A u_h, v\rangle=(h,v)_{L^2}, \text{ for all }v\in H_0^1(\Omega),
\end{equation*}
and thus $\left( u_h, v\right)_{H_0^1}=(h,v)_{L^2}$.

If we identify $H^{-1}(\Omega)$ with $H_0^1(\Omega)$, the operator $L$ induces in $H_0^1(\Omega)$ the scalar product $(\cdot,\cdot)_A$ and the norm $|\cdot|_{A}$, given by \begin{equation*}
     \left( u, v\right)_A:=\langle  A u, v \rangle 
 \end{equation*} 
 and \begin{equation*}
     |u|_A:=\sqrt{\langle Au, u \rangle  },
 \end{equation*}
 for all $u,v \in H_0^1(\Omega)$.
 
 From the strong monotony of $A$ given in \eqref{tare monotonie L}, we immediately deduce the following  \textit{Poincaré  inequality}  \begin{equation}\label{poincare}
     |u|_{L^2}\leq\sqrt{\theta} |u|_A, \text{ for all }u\in H_0^1(\Omega).
 \end{equation}

\subsection{Matrices convergent to zero}

A square matrix $M\in \mathcal{M}_{n\times n}\left(\mathbb{R}+\right)$ is considered to be "convergent to zero" if its power $M^k$ tends to the zero matrix as $k\rightarrow \infty$.
Other equivalent characterizations include the requirement that the spectral radius of the matrix is less than one, or if the inverse of $I-A$ (where $I$ is the identity matrix) is both invertible and has nonnegative entries (see, e.g., \cite{nonnegative}).

The following result, concerning matrices convergent to zero, holds true: \begin{lemma}[{\cite[Lemma~2.2]{s}}]
\label{first_lemma} Let $\left( x_{k,p}\right) _{k\geq 1},\ \left(
y_{k,p}\right) _{k\geq 1}\ \ $ be two sequences of vectors in $%
\mathbb{R}
_{+}^{n}$ (column vectors) depending on a parameter $p,$ such that
\begin{equation*}
x_{k,p}\leq A x_{k-1,p}+y_{k,p}
\end{equation*}%
for all $k$ and $p,$ where $A\in \mathbb{M}_{n \times n}(\mathbb{R}_{+})$ is a matrix
convergent to zero. If the sequence $\left( x_{k,p}\right)
_{k\geq 1}$ is bounded uniformly with respect to $p$ and $y_{k,p}\rightarrow
0_{n}$ as $k\rightarrow \infty $ uniformly with respect to $p,$ then $%
x_{k,p}\rightarrow 0_{n}$ as $k\rightarrow \infty $ uniformly with respect
to $p.$
\end{lemma}

\subsection{Stationary Stokes-type equation}

Let $\Omega^\prime\subset \mathbb{R}^N$ $(N\leq 3)$ be an open and bounded domain and let
 $\textbf{f}\in H^{-1}(\Omega^\prime)^N$. We recall some results related to the  Stokes-type problem (see, e.g., \cite{KP, david}), \begin{equation}\label{sistem_simplu_stokes}
    \begin{cases}
        -\Delta \textbf{v}+\mu \textbf{v}+\nabla p=\textbf{f} \text{ in }\Omega^\prime
        \\
        \operatorname{div} \textbf{v}=0 \text{ in }\Omega^\prime\\
        \textbf{v}=0 \text{ on }\Omega^\prime.
    \end{cases}
\end{equation}
A solution is sought in the Sobolev space 
 \begin{equation*}
    V=\left\{ \textbf{v}\in H_0^1(\Omega^\prime)^N\, : \,  \operatorname{div} \textbf{v}=0 \right\}.
\end{equation*}
We endow $V$ with the scalar product \begin{equation*}
    (\textbf{v},\textbf{w})_V=\int_\Omega \nabla \textbf{v} \cdot \nabla \textbf{w}+\int_\Omega \mu \textbf{v} \cdot \textbf{w} 
\end{equation*}
and the corresponding norm $|\textbf{v}|_V=\sqrt{    (\textbf{v},\textbf{v})_V}$. 
One  has the Poincare's inequality (see, e.g., \cite{KP}), \begin{equation*}
    |\textbf{v}|_{(L^2)^N}\leq \frac{1}{\lambda_1+\mu}|\textbf{v}|_{V}, \text{ for all } \textbf{v}\in V,
\end{equation*}
where $\lambda_1$ is the first eigenvalue of the Dirichlet problem $-\Delta \textbf{v}=\lambda \textbf{v}$ in $\Omega^\prime$ and $\textbf{v}=0$ on $\partial \Omega^\prime$.

For $(\textbf{v},p)\in H_0^1(\Omega^\prime)^N\times L^2(\Omega)$, the variational formulation of the system \eqref{sistem_simplu_stokes} is: \begin{equation*}
    (\textbf{v},\textbf{w})_{(H_0^1)^N}+\mu  (\textbf{v},\textbf{w})_{(L^2)^N}- (p,\operatorname{div}\textbf{w})_{L^2}=\langle \textbf{f},\textbf{w}\rangle, \text{ for all }\textbf{w}\in H_0^1(\Omega^\prime)^N
\end{equation*}

If $\textbf{v}\in V$,  the above relation becomes, \begin{equation}\label{sistem_fara_p}
     (\textbf{v},\textbf{w})_{V}=\langle \textbf{f},\textbf{w}\rangle, \text{ for all }\textbf{w}\in V.
\end{equation}
 Here, $\langle \cdot, \cdot\rangle$ stands for the dual pairing between $V^\prime$ and $V$.
\begin{remark}
   If we find a solution $\textbf{v}\in V$ to \eqref{sistem_fara_p}, the pressure $p\in L^2(\Omega^\prime)$ is guaranteed by De Rham's Theorem (see, e.g., \cite{fara_p, fara_p2}). 
\end{remark}

From Riesz's representation theorem, there exists a unique weak solution $\textbf{v}_{\textbf{f}}\in V$ of the problem \eqref{sistem_fara_p}, that is, there is only one $\textbf{v}_{\textbf{f}}\in V$ such that \begin{equation*}
    \left( \textbf{v}_{\textbf{f}},\textbf{w}\right)_V=\langle\textbf{f},\textbf{w}\rangle,\text{ for all }\textbf{w}\in V.
\end{equation*}
Moreover, one has the inequality,
\begin{equation*}
    |\textbf{v}_{\textbf{f}}|_V^2=(\textbf{f},\textbf{v}_{\textbf{f}})\leq |\textbf{f}|_{V^\prime} \left| \textbf{v}_{\textbf{f}}\right|_V,
\end{equation*}
i.e., $ |\textbf{v}_{\textbf{f}}|_V\leq |\textbf{f}|_{V^\prime}.$

Thus, we may define the solution operator $S\colon V^\prime \to V$, $S(\textbf{f})=\textbf{v}_{\textbf{f}}$. Clearly, it is an isomorphism between $V^\prime $ and $ V$.





\section{Main results}

Let $H$ be a Hilbert space together with the scalar product $\left(\cdot, \cdot \right)_H$ and the induced norm $|\cdot|_H$. 
We consider the system of the type \begin{equation} \label{sistem abstract desfacut}
    \begin{cases}
        u=N_u(u,v) \\
        -v=N_v(u,v),
    \end{cases}
\end{equation}
where $N\colon H\times H \to \mathbb{R}$ is a continuous operator.
\begin{remark}
The structure of the system \eqref{sistem abstract desfacut} that we have considered situates it as an intermediary between gradient-type systems and Hamiltonian systems.  Clearly, it admits a variational structure given by the functional \begin{equation*}
    E(u,v)=\tfrac{1}{2}|u|_H^2-\tfrac{1}{2} |v|_{H}^2-N(u,v).
\end{equation*}
However, in general, this functional is unbounded from both above and below.
\end{remark}
 To the system \eqref{sistem abstract desfacut}, we associate the partial functionals $E_1, E_2 \colon H\times H \to \mathbb{R}$ given by  \begin{equation*}
    E_1(u,v)=\tfrac{1}{2}|u|_H^2-N(u,v),  
\end{equation*}and
\begin{equation*}
    E_2(u,v)=-\tfrac{1}{2}|v|_H^2-N(u,v).
\end{equation*}
One easily sees that both $E_1$ and $E_2$ are Fréchet differentiable and  moreover, \begin{align}
   \label{definitie E11 si E22}
 &  E_{11}(u,v):=(E_1)_u=u-N_u(u,v), \\ & E_{22}(u,v):=(E_2)_v=-v-N_v(u,v). \notag
\end{align}
We say that an point $(u^\ast,v^\ast)\in H \times H$ is a \textit{partial critical point} for the pair of functionals $(E_1, E_2)$ if it satisfies \begin{equation*}
    E_{11}(u^\ast,v^\ast)=0 \text{ and } E_{22}(u^\ast,v^\ast)=0.
\end{equation*}
Obviously, any partial critical point for the pair of functionals $(E_1, E_2)$ is a solution to the system \eqref{sistem abstract desfacut}.

The subsequent result establishes a relation between the critical points of the functional $E$ and the partial critical points of the pair of functionals $(E_1,E_2)$. \begin{lemma}\label{lema}
   A pair $(u^\ast, v^\ast)\in H \times H$ is a critical point of $E$ if and only if it is a partial critical point for the pair of functionals $(E_1, E_2)$.
\end{lemma}
\begin{proof}
 The result is immediate if we observe that for any $u, v \in H$, the following relations hold: \begin{equation*}
      E_u(u,v)=u-N_u(u,v)=E_{11}(u,v)
  \end{equation*}
  and \begin{equation*}
      E_v(u,v)=-v-N_v(u,v)=E_{22}(u,v).
  \end{equation*}
\end{proof}

\subsection{Existence of a partial critical point}

Now we are prepared to present our main result, which essentially involves establishing sufficient conditions to ensure the existence of at least one partial critical point for the pair of functionals $(E_1, E_2)$.


\begin{theorem}\label{teorema principala}
     Under the previous established setting, we additionally assume: \begin{itemize}
         \item[(h1)] One has the growth conditions \begin{equation}\label{conditii crestere}
             -\underline{\alpha} |v|_H^2-C\leq N(u,v)\leq \overline{\alpha}|u|_H^2+C, \text{ for all }u,v\in H,
         \end{equation}
where $0\leq \overline{\alpha},\underline{\alpha}<\frac{1}{2}$ such that $\overline{\alpha}+\underline{\alpha}<\frac{1}{2}$ and $C>0$.
         \item[(h2)] There are nonegative real numbers $m_{ij} \, (i,j \in \{1,2\})$ such that the following monotony conditions hold true: \begin{equation*}
             \left( N_u(u,v)-N_u(\overline{u},\overline{v}), u-\overline{u}\right) \leq m_{11} |u-\overline{u}|^2_H+m_{12} |u-\overline{u}|_H |v-\overline{v}|_H,
         \end{equation*}
         \begin{equation*}
             \left( N_v(u,v)-N_v(\overline{u},\overline{v}), v-\overline{v}\right) \geq -m_{22} |v-\overline{v}|^2_H-m_{21} |u-\overline{u}|_H |v-\overline{v}|_H,
         \end{equation*}
         for all $u,v,\overline{u},\overline{v}\in H$.
         \item[(h3)] The matrix $M=(m_{ij})_{1\leq i,j\leq 2}$ is convergent to zero.
     \end{itemize}
     Then, there exists a partial critical point $(u^\ast,v^\ast)\in H\times H$ for the pair of functionals $(E_1,E_2)$.    
\end{theorem}
\begin{proof}

For better comprehension, we structure our proof into several steps.

\bigskip
\textbf{Step 1: } Boundedness from below and upper of the functionals $E_1, E_2$.
\bigskip

Let $u,v\in H$. The growth conditions  \eqref{conditii crestere} yields \begin{align*}
    E_1(u,v)&=\tfrac{1}{2}|u|_H^2-N(u,v)\\&
    \geq \left(\tfrac{1}{2}-\overline{\alpha} \right)|u|^2_H-C \geq -C,
\end{align*}
and \begin{align*}
    E_2(u,v)&=-\tfrac{1}{2}|v|_H^2-N(u,v)\\&
    \leq -\left(\tfrac{1}{2}-\underline{\alpha} \right)|v|^2_H+C \leq C.
\end{align*}

\bigskip
\textbf{Step 2: } Construction of an approximation sequence $(u_k,v_k)$.
\bigskip

We employ a method similar to the one described in \cite{p}. For an $v_0$ arbitrarily chosen,  using Ekeland's variational principle within a recursive procedure, we generate a sequence $(u_k,v_k)\in H \times H$ such that \begin{align}
   & E_{1}(u_{k},v_{k-1}) \leq \inf_{H}E_{1}(\cdot ,v_{k-1})+\tfrac{1}{k}\text{ },\ \ \ \text{ }E_{2}(u_{k},v_{k})%
\geq \sup_{H}E_{2}(u_{k},\cdot )-\tfrac{1}{k},
\label{inf sup ineq} 
 \\&
\left\vert E_{11}(u_{k},v_{k-1})\right\vert _{H} \leq \tfrac{1}{k},%
\text{ \ \ \ }\left\vert E_{22}(u_{k},v_{k})\right\vert _{H}\leq
\tfrac{1}{k}.\text{ }  \label{derivative_ineq}
\end{align}
\bigskip
\textbf{Step 3: } Boundedness of the sequence $u_k$.
\bigskip

 From \eqref{conditii crestere} and the second relation from \eqref{inf sup ineq},  we infer\begin{align*}
   \tfrac{1}{2}|u_{k}|^2_H  &\leq N\left( u_k,v_{k-1}\right)+ \inf _HE_1(\cdot,v_{k-1})+\tfrac{1}{k} 
  \\& \leq  N\left( u_k,v_{k-1}\right)+ E_1(0,v_{k-1})+1
   \\& \leq \overline{\alpha}|u_{k}|^2_H +\underline{\alpha} |v_{k-1}|^2_H+2C+1.
\end{align*}
Hence, \begin{equation}\label{marginie uk}
    |u_k|_H^2\leq \frac{\underline{\alpha}}{\tfrac{1}{2}-\overline{\alpha}}|v_{k-1}|_H^2+C_1,
\end{equation}
for some constant $C_1$.
Under similar computations, from the second relation of \eqref{inf sup ineq} we obtain \begin{equation*}
      \tfrac{1}{2}|v_{k}|^2_H \leq   \overline{\alpha}|u_{k}|^2_H +\underline{\alpha} |v_{k}|^2_H+2C+1,
\end{equation*}
which yields \begin{equation}\label{marginie vk}
     |v_k|_H^2\leq \frac{\overline{\alpha}}{\tfrac{1}{2}-\underline{\alpha}}|u_{k}|_H^2+C_2,
\end{equation}
for some constant $C_2$.
Now, we combine inequalities \eqref{marginie uk} and \eqref{marginie vk} to deduce \begin{equation*}
     |u_k|_H^2\leq \mu |u_{k-1}|^2_H+C_3,
\end{equation*}
where \begin{equation*}
    \mu=\frac{\overline{\alpha}\, \underline{\alpha}}{
     \left(\tfrac{1}{2}-\overline{\alpha}\right)\left(\tfrac{1}{2}-\underline{\alpha}\right)}.
\end{equation*}
From (h1), we easily see that $\mu<1$, which guarantees that $u_k$ is bounded.

\bigskip

\textbf{Step 4: } Convergence of the sequences $u_k$ and $v_k$.
\bigskip

Let $p>0$.  From the monotony conditions (h2), we have \begin{align*}
    |u_{k+p}-u_k|_H^2&=\left( u_{k+p}-N_u(u_{k+p},v_{k+p-1}) -u_k+N_u(u_{k},v_{k-1}),u_{k+p}-u_k  \right)_H\\& \quad +\left(N_u(u_{k+p},v_{k+p-1})- N_u(u_{k},v_{k-1}),u_{k+p}-u_k\right)_H
    \\& \leq \left(\frac{1}{k+p} +\frac{1}{k}\right)   |u_{k+p}-u_k|_H+m_{11}\left|    u_{k+p}-u_k\right|_H^2\\& \quad +m_{11}\left|    u_{k+p}-u_k\right|_H \left|    v_{k+p-1}-v_{k-1}\right|_H.
\end{align*}
Thus, \begin{equation} \label{ukp minus uk}
    |u_{k+p}-u_k|_H\leq \tfrac{2}{k}+m_{11}\left|    u_{k+p}-u_k\right|_H+m_{11}\left|v_{k+p-1}-v_{k-1}\right|_H.
\end{equation}
For the sequence $(v_k)$, we similarly obtain
\begin{align*}
    |v_{k+p}-v_k|_H^2&=\left(v_{k+p}-v_k, -v_{k}-N_v(u_k,v_k)+v_{k+p}+N_v(u_{k+p},v_{k+p}) \right)_H\\& \quad -\left(v_{k+p}-v_k, N_v(u_{k+p},v_{k+p})-N_v(u_k,v_k) \right)_H\\&
    \leq  |v_{k+p}-v_k|_H\left(\frac{1}{k+p} +\frac{1}{k}\right)+m_{11} |v_{k+p}-v_k|_H^2\\& \quad+m_{11} |v_{k+p}-v_k|_H  |u_{k+p}-u_k|_H^2.
\end{align*}
Hence,  \begin{equation} \label{vkp minus vk}
    |v_{k+p}-v_k|_H\leq \frac{2}{k}+m_{11}\left|   v_{k+p}-v_k\right|_H+m_{11}\left|u_{k+p}-u_{k}\right|_H.
\end{equation}
If we write the  relations \eqref{ukp minus uk} and \eqref{vkp minus vk} in matrix form, we infer \begin{equation*}
    \begin{bmatrix}
         |u_{k+p}-u_k|_H \\    |v_{k+p}-v_k|_H
    \end{bmatrix}
    \leq \begin{bmatrix}
        m_{11} & 0 \\ m_{11} & m_{11}
    \end{bmatrix} \begin{bmatrix}
         |u_{k+p}-u_k|_H \\    |v_{k+p}-v_k|_H
    \end{bmatrix}+\begin{bmatrix}
       0 & m_{11} \\ 0 & 0
    \end{bmatrix} \begin{bmatrix}
         |u_{k+p-1}-u_{k-1}|_H \\    |v_{k+p-1}-v_{k-1}|_H
    \end{bmatrix}+\begin{bmatrix}
       \frac{2}{k} \\    0    \end{bmatrix}.
\end{equation*}

Since $u_k$ is bounded and the matrix $M$ converges to zero, we can conclude from \Cref{first_lemma} that both $u_k$ and $v_k$ converge. Let us denote their limits as $u^\ast$ and $v^\ast$.

\bigskip
\textbf{Step 5: }Passing to limit.
\bigskip

Since $u_k\to u^\ast$ and $v_k\to v^\ast$, the conclusion follows immediately if we pass to limit in  \eqref{derivative_ineq}.
\end{proof}

\begin{remark}
The partial critical point obtained in \Cref{teorema principala} has the additional  property of being a Nash equilibrium for the functionals $E_1$ and $-E_2$ (see, e.g., \cite{p, s_jnfa} for further details on Nash equilibrium). This relationship is a result of taking the limit in \eqref{inf sup ineq}, which gives \begin{align*}
       &  E_1(u^\ast,v^\ast)=\inf_{H} E_1(\cdot,v^\ast),
       \\& E_{2}(u^\ast,v^\ast)=\sup_{H} E_2(u^\ast,\cdot).
     \end{align*}
\end{remark}

\subsection{Relation with the classical mountain pass approach}

The well-known approach to obtain critical points for functionals that lack upper or lower bounds is to employ the Ambrosetti-Rabinowitz results, which guarantee the existence of mountain pass points (as seen in \cite[Theorem~2.1]{ar}). The typical conditions imposed on the functional $E$ are: \begin{itemize}
        \item[(I1)] There exists $\tau>0$ such that \begin{equation*}
            E(u,v)\geq \alpha >E(0,0), \text{ for all }|(u,v)|_{H \times H}=\tau.
        \end{equation*}
        \item[(I2)] There exists $e\in H\times H$ with $|e|>\tau$ such that \begin{equation*}
            E(e)<\inf_{|(u,v)|=\tau} E(u,v) .
        \end{equation*}
        \item[(I3)] The functional $E$ has the Palais-Smale property, i.e., if $e_k$ is a sequence such that \begin{equation*}
            E(e_k) \text{ is bounded}
        \end{equation*}
        and \begin{equation*}
            \nabla E(e_k)\to 0,
        \end{equation*}
        then $e_k$ admits a convergent subsequence.
    \end{itemize}
    
   In the following, we will explore how these conditions align with our hypotheses (h1)-(h3).
  
\textit{Condition (I1):}

Let $(u,v)\in H\times H$ such that $|(u,v)|_{H \times H}=\tau$, i.e.,  $|u|_H+|v|_H=\tau$. We compute, 
    \begin{equation*}
          E(u,v)=\tfrac{1}{2} |u|_H^2-\tfrac{1}{2}|v|_H^2- N(u,v)= \tfrac{1}{2}\left(|u|_H+|v|_H \right)\left(|u|_H-|v|_H \right) -N(u,v) .
    \end{equation*}
Thus, for the relation $E(u,v)\geq \alpha$ to hold, we need   \begin{equation}\label{conditie N}
       N(u,v)<\tfrac{\tau}{2}\left(|u|_H-|v|_H \right)-\alpha, \text{ for all }|(u,v)|_{H \times H}=\tau.
    \end{equation}
    On the other hand, $ E(0,0)<\alpha$  implies that \begin{equation*}
       -N(0,0)< \alpha, \text{ i.e., }-\alpha< N(0,0).
    \end{equation*}
Hence, relation \eqref{conditie N} becomes \begin{equation*}
     N(u,v)<-\tfrac{\tau}{2}\left(|u|_H-|v|_H \right)+N(0,0), \text{ for all }|(u,v)|_{H \times H}=\tau,
\end{equation*}
that is \begin{equation}\label{conditie N2}
     N(u,v)-N(0,0)< \tfrac{\tau}{2}\left(|u|_H-|v|_H \right), \text{ for all }|(u,v)|_{H \times H}=\tau.
\end{equation}

In our main result such a condition is not required, which enables us to encompass a broader range of situations in which the system \eqref{sistem abstract desfacut} is solvable. It is clear that there might be cases where our result is not applicable, but the Ambrosetti-Rabinowitz theorem is, and vice versa.

\textit{Condition (I2)}.

This condition is satisfied; for instance, one can take $(0, \gamma e)$, where $\gamma$ is a sufficiently large real number, and $e$ is a fixed element from $H$ distinct from the origin of the space. Indeed, \begin{align*}
    E(0,\gamma e)& =-\tfrac{\gamma^2}{2} |e|^2-N(0,\gamma e)\\& \leq -\left(\tfrac{1}{2}-\underline{\alpha} \right) \gamma^2 |e|^2+C \to -\infty, \text{ as }\gamma\to \infty.
\end{align*}

\textit{Condition (I3)}

Let $e_k=( (e_1)_k, (e_2)_k)$ be a sequence such that \begin{equation*}
E(e_k) \text{ is uniformly bounded},
\end{equation*} 
and $\nabla E(e_k)\to 0$, i.e., \begin{align}
& \label{derivata E u} \, \, \, \, \, \, \, (e_1)_k-N_u(e_k)\to 0, \\& 
 - (e_2)_k-N_v(e_k)\to 0.\label{derivata E v}
\end{align}
Let $k_0$ large enough such that $|(e_1)_k-N_u(e_k)|\leq 1$, for all $k\geq k_0$. Consequently, when taking a scalar product in \eqref{derivata E u} with $(e_1)_k$ for $k\geq k_0$, we obtain  \begin{align*}
    \left((e_1)_k-N_u\left(e_k\right) ,  (e_1)_k \right)_H \leq &  \left| (e_1)_k\right|_H .
\end{align*}
From the monotony conditions (h2) we deduce \begin{align*}
    \left((e_1)_k-N_u\left(e_k\right) ,  (e_1)_k \right)_H&=  \left((e_1)_k, (e_2)_k \right)_H-\left(N_u(e_k), (e_1)_k \right)_H\\& =
    \left| (e_1)_k\right|_H^2-\left(N_u(e_k)-N_u(0), (e_1)_k \right)_H-\left(N_u(0), (e_1)_k \right)_H \\& \geq
     (1-m_{11})|(e_1)_k|^2_H-m_{11}|(e_1)_k|_H|(e_2)_k|_H-|N_u(0,0)|_H|(e_1)_k|_H.
\end{align*}
Hence,  \begin{equation}\label{r1}
    (1-m_{11})|(e_1)_k|^2_H-m_{11}|(e_1)_k|_H|(e_2)_k|_H\leq  \left( |N_u(0,0)|_H+1\right)|(e_1)_k|_H.
\end{equation}
Following a similar reasoning, from \eqref{derivata E v} we have
\begin{equation}\label{r2}
     (1-m_{11})|(e_2)_k|^2_H-m_{11}|(e_1)_k|_H|(e_2)_k|_H\leq  \left( |N_v(0,0)|_H+1\right)|(e_2)_k|_H.
\end{equation}
Therefore,  the above two relations \eqref{r1}, \eqref{r2} yields \begin{equation*}
    \beta |(e_2)_k|_H \leq D,
\end{equation*} 
where $D$ is some constant and \begin{equation*}
    \beta=  1-m_{11}- \frac{m_{11}m_{11}}{1-m_{11}}.
\end{equation*}Given that the matrix $M$ is convergent to zero, we immediately deduce that \begin{equation*}
  \beta= 1-m_{11}- \frac{m_{11}m_{11}}{1-m_{11}}>0,
\end{equation*}
which guarantees the boundedness of $(e_2)_k$. From this, is clear that $(e_1)_k$ is also bounded.

The boundedness of the sequence $e_k$ guarantees the existence of a  weakly convergent subsequence. However, establishing the strong convergence of this subsequence solely under hypotheses (h1)-(h3), remains an open question. Thus, we can formulate the following problem: 

\begin{question*}
    Given only the assumptions (h1)-(h3), does the functional $E$ satisfy the Palais-Smale condition?
\end{question*}

Nonetheless, under certain additional assumptions, this result is valid.
\begin{theorem}
    Assume that the operator $K:=\nabla N=(N_u, N_v)$ is compact. Then the functional $E$ satisfies the Palais-Smale condition. 
\end{theorem}
\begin{proof}
   
Note that $\nabla E=I-K$. Given the compactness of $K$ and the boundedness of $e_k$, it follows that there exists a  subsequence, also denoted as $e_k$, such that $K(e_k)$ converges to a point $\Tilde{e}$ in $H \times H$. Thus, \begin{equation*}
        |e_k-\Tilde{e}|-|\Tilde{e}-K(e_k)|\leq |e_k-K(e_k)|=|\nabla E(e_k)|.
    \end{equation*}
    Now, the conclusion is immediate since $|\nabla E(e_k)|\to 0$ and $|K(e_k)-\Tilde{e}|\to 0.$ 
\end{proof}
      
\section{Applications}
In this section, we present two application for the results obtained in \Cref{teorema principala}.
\subsection{Abstract system on $H_0^1(\Omega)$ }

Let us consider the Dirichlet problem \begin{equation}\label{sistem aplicatie}
    \begin{cases}
        Au=F_u(u,v) \\
        -Av=F_y(u,v) \\
        u|_{\partial \Omega}=  v|_{\partial \Omega}=0,
    \end{cases}
\end{equation}
where $\Omega\subset \mathbb{R}^n \, (n\geq 3)$ is a bounded open set, $F\colon \mathbb{R}^2\to \mathbb{R}$ is a $C^1$ functional and the operator $A$ is defined in \Cref{definitie A}.  Here, $ F_u$ and $F_v$ stand for the partial derivatives of $F$ with respect to the first and second component, respectively.  
We use $(\cdot, \cdot)$ and $|\cdot|$  to denote the scalar product and the corresponding norm  in $\mathbb{R}^2$.

The Hilbert space $H$ is considered to be the Sobolev space $H_0^1(\Omega)$ equipped with the scalar product  $(\cdot,\cdot)_A$ and the corresponding norm  $|\cdot|_A$.
Clearly, the system \eqref{sistem aplicatie} admits a variational given   by the energy functional $E\colon H_0^1(\Omega)\times H_0^1(\Omega)\to \mathbb{R}$,  \begin{equation*}
    E(u,v)=\tfrac{1}{2}|u|_A^2-\tfrac{1}{2}|v|^2_A-\int_{\Omega} F(u,v).
\end{equation*}
The partial functionals $E_1, E_2\colon H_0^1(\Omega)\times H_0^1(\Omega)\to \mathbb{R}$ associated to the system \eqref{sistem aplicatie} are given by  \begin{align*}
    & E_1(u,v)=\tfrac{1}{2}|u|_A^2-\int_\Omega F(u,v) ,\\&
    E_2(u,v)=-\tfrac{1}{2}|v|_A^2-\int_\Omega F(u,v).
\end{align*}
If we denote \begin{equation*}
    \begin{cases}
       f_1(u,v)=F_u(u,v )\\
       f_2(u,v= F_v(u,v),
    \end{cases}
\end{equation*}
the identification of $H^{-1}(\Omega)$ with $H_0^1(\Omega)$ via $A^{-1}$,  yields to the representation \begin{equation*}
    \nabla E(u,v)=\left( u- A^{-1} f_1(u,v), v-A^{-1}f_2(u,v)\right)=\left(E_{11}(u,v), E_{22}(u,v)\right),
\end{equation*}
where $E_{11}, E_{22}$ stand for the partial Fréchet derivatives of $E_1$ and $E_2$ with respect to the first and second component, respectively.
Consequently, the operator $N$ is given by \begin{equation*}
    N(u,v)=\int_\Omega F(u,v) 
\end{equation*}
and its derivatives are the  Nemytskii’s operators \begin{equation*}
    N_u(u,v)=A^{-1} f_1(u,v) \text{ and } N_v(u,v)=A^{-1} f_2(u,v).
\end{equation*}
On the potential $F$, we assume the following conditions: \begin{itemize}
    \item[(H1)] There exist real numbers $0 \leq \tau_1, \tau_2 \leq \frac{1}{4 \theta}$  and $C>0$, such that the following conditions hold  \begin{equation*}
      -\tau_1|x|^2-C\leq F(x,y) \leq \tau_2 |y|^2 +C, \text{ for all }x,y\in \mathbb{R}^2.
        \end{equation*}
\end{itemize} 

Related to the gradient of $F$, let us assume: \begin{itemize}
    \item[(H2)] There are nonegative real numbers $\overline{m}_{ij}$ such that,  for all $x,y\in \mathbb{R}^2$, one has  the monotony conditions: \begin{align*}
       \left( f_1(x,y)-f_1(\overline{x},\overline{y}),x-\overline{x}\right) \leq \overline{m}_{11}|x-\overline{x}|^2+ \overline{m}_{12}|x-\overline{x}||y-\overline{y}|
    \end{align*}
    and \begin{equation*}
          \left( f_2(x,y)-f_2(\overline{x},\overline{y}),x-\overline{x}\right) \geq -\overline{m}_{22}|y-\overline{y}|^2- \overline{m}_{21}|x-\overline{x}||y-\overline{y}|.
    \end{equation*}
    
\end{itemize}

    Finally,  the constants specified in (H2) are such that:
    \begin{itemize}
    \item[(H3)] The matrix \begin{equation*}
     M:=  \theta\begin{bmatrix}
            \overline{m}_{11} & \overline{m}_{12} \\
            \overline{m}_{21} & \overline{m}_{22}
        \end{bmatrix}
    \end{equation*} 
    is convergent to zero.
\end{itemize}

In the subsequent, we prove that conditions (H1)-(H3) are sufficient to ensure the existence of a partial critical point for the pair of functionals $(E_1, E_2)$.
 \begin{theorem}
    Assume (H1)-(H3) hold true. Then, there exists a pair of points $(u^\ast, v^\ast)\in H_0^1(\Omega) \times H_0^1(\Omega)$ such that it is a critical point for the functional $E$.
    
    Furthermore, it has the additional property that
    \begin{align*}
    &    E_1(u^\ast, v^\ast)=\inf_{H_0^1(\Omega)} E_1(\cdot, v^\ast),
    \\&   E_2(u^\ast, v^\ast)=\sup_{H_0^1(\Omega)} E_2(u^\ast, \cdot).
    \end{align*}

\end{theorem}
\begin{proof}
   We verify that all conditions from \Cref{teorema principala} are satisfied.

\textit{Check of the condition (h1)}. 
Let $u,v \in H_0^1(\Omega)$. Then, for some constant $C_1>0$, using the Poincaré inequality \eqref{poincare}, we deduce \begin{align*}
    N(u,v)=\int_\Omega F(u,v)
    &\leq \tau_2 |u|_{L^2}^2+C_1\\&
    \leq \tau_2 \theta |u|_A^2+C_1,
\end{align*}
and \begin{align*}
    N(u,v)=\int_\Omega F(u,v) 
    &\geq  -\tau_1 |v|_{L^2}^2-C_1\\& 
    \geq -\theta \tau_1 |u|_A^2-C_1.
\end{align*}
The conclusion is immediate since $\theta \, \tau_1<\frac{1}{4}$ and $ \theta \,\tau_2 <\frac{1}{4}$.

\textit{Check of the condition (h2)}.
    For any $u,v,\overline{u}, \overline{v}\in H_0^1(\Omega)$, one has \begin{align*}
        \left( N_u(u,v)-N_u(\overline{u}, \overline{v}), u-\overline{u}\right)_A&=\left( A^{-1} f_1(u,v)-f_1(\overline{u}, \overline{v}), u-\overline{u}\right)_A
        \\&
       = \left( f_1(u,v)-f_1(\overline{u}, \overline{v}), u-\overline{u} \right)_{L^2}\\& \leq \overline{m}_{11}|u-\overline{u}|_{L^2}^2+\overline{m}_{12}|u-\overline{u}|_{L^2} |v-\overline{v}|_{L^2}
    \end{align*}
    From the Poincaré inequality \eqref{poincare}, we further obtain \begin{align*}
         \left( N_u(u,v)-N_u(\overline{u}, \overline{v}), u-\overline{u}\right)_A\leq \theta \, \overline{m}_{11}|u-\overline{u}|_A^2+\theta \, \overline{m}_{12}|u-\overline{u}|_A |v-\overline{v}|_A.
    \end{align*}
    Similar estimates are obtained for $N_v$, \begin{align*}
        \left( N_v(u,v)-N_v(\overline{u}, \overline{v}), u-\overline{u}\right)_A& 
        = \left( f_2(u,v)-f_2(\overline{u}, \overline{v}), u-\overline{u} \right)_{L^2}
      \\&   \geq -\overline{m}_{22}|v-\overline{v}|_{L^2}^2-\overline{m}_{21}|u-\overline{u}|_{L^2}|v-\overline{v}|_{L^2}\\&
      \geq -\theta \, \overline{m}_{22}|v-\overline{v}|_A^2-\theta \, \overline{m}_{21}|u-\overline{u}|_A |v-\overline{v}|_A.
    \end{align*}
    Consequently, condition (h2) is satisfied with $m_{ij}=\theta \, \overline{m}_{ij}$, $(i,j=\{1,2\})$.

    \textit{Check of the condition (h3)}. This condition is immediate from (H3).

   Thus, all hypothesis of \Cref{teorema principala} are satisfied and consequently, there exists a partial critical point $(u^\ast, v^\ast)$ for the pair of functionals $(E_1, E_2)$ such that     \begin{align*}
    &    E_1(u^\ast, v^\ast)=\inf_{H_0^1(\Omega)} E_1(\cdot, v^\ast),
    \\&   E_2(u^\ast, v^\ast)=\sup_{H_0^1(\Omega)} E_2(u^\ast, \cdot).
    \end{align*} Moreover, from \Cref{lema}, the pair $(u^\ast, v^\ast)$ is a critical point for the functional $E$. 
\end{proof}

\subsection{Stokes-type coupled system}

We consider the Stokes-type coupled system  \begin{equation}\label{stokes}
    \begin{cases}
        -\Delta \textbf{u}_1+\mu \textbf{u}_1+\nabla p_1= F_{\textbf{u}_1}(\textbf{u}_1,\textbf{u}_2) \text{ in }\Omega^\prime
        \\
            -\Delta \textbf{u}_2+\mu \textbf{u}_2+\nabla p_2=-F_{\textbf{u}_2}(\textbf{u}_1,\textbf{u}_2) \text{ in }\Omega^\prime
        \\
        \operatorname{div} \textbf{u}_i=0  \text{ in }\Omega^\prime\\
        \textbf{u}_i=0  \quad (i=1,2)\text{ on }\partial\Omega^\prime,
    \end{cases}
\end{equation}
where $\mu>0$ and $F\colon \mathbb{R}^{2N} \to \mathbb{R}$  is a $C^1$ functional. Here, $ F_{\textbf{u}_1},  F_{\textbf{u}_2}$ represent for the partial derivatives of $F$ with respect to the first and second component, respectively.

Our problem \eqref{stokes}  is equivalent with the fixed point equation  \begin{equation}\label{stokes fara presiuni punct fix}
    \begin{cases}
        \textbf{u}_1=S^{-1} F_{ \textbf{u}_1} \\
         \textbf{u}_2=S^{-1} F_{ \textbf{u}_2},
    \end{cases}
\end{equation}
where $\left(  \textbf{u}_1,  \textbf{u}_2\right)\in V \times V$.

Now, we can apply \Cref{teorema principala}, where $H=V$ and \begin{equation*}
    N\left(  \textbf{u}_1,  \textbf{u}_2\right)=\int_{\Omega^\prime} F\left(  \textbf{u}_1,  \textbf{u}_2\right).
\end{equation*}
The verification of conditions (h1)-(h3) follows a similar  process to the previous application. This is done under the assumption that $F$ satisfies (H1)-(H3), where by $(\cdot,\cdot)$ and $|\cdot|$ we understand the usual scalar product and norm in  $\mathbb{R}^N$, while $\theta$ is replaced by $\frac{1}{\lambda_1+\mu}$.

Therefore, \Cref{teorema principala} ensures the existence of a pair $({\bf u}_1^\ast,{\bf u}_2^\ast)\in V\times V$, which, according to De Rham’s Lemma, further guarantees the pressures $(p_1,p_2)\in L^2(\Omega^\prime)\times L^2(\Omega^\prime)$ such that $(({\bf u}_1^\ast,p_2),({\bf u}_2^\ast,p_2))\in \left( V\times L^2(\Omega^\prime)\right)^2$ solves the system \eqref{stokes}.

\section*{Acknowledgments}

This work was supported by the project "Nonlinear Studies of Stratified Oceanic and Atmospheric Flows " funded by the European Union – NextgenerationEU and Romanian Government, under National Recovery and Resilience Plan for Romania, contract no. 760040/23.05.2023, cod PNRR-C9-I8-CF 185/22.11.2022, through the Romanian Ministry of Research, Innovation and Digitalization, within Component 9, Investment I8.

  \bibliographystyle{elsarticle-harv} 
  \bibliography{references}



\end{document}